\documentclass[10pt,A4, leqno]{amsart}

\usepackage{enumerate}
\usepackage[dvipdfmx]{graphicx} 
\usepackage[mathscr]{eucal}
\usepackage{mathrsfs}
\usepackage{amscd}
\usepackage{physics}
\usepackage{amsfonts}
\usepackage{mathtools}
\usepackage{amsmath}
\usepackage{amssymb}
\usepackage{amsthm}
\usepackage{latexsym}
\usepackage[dvipdfmx]{hyperref}

\makeatletter
\@namedef{subjclassname@2020}{\textup{2020} Mathematics Subject Classification}
\makeatletter

\numberwithin{equation}{section}

\theoremstyle{plain}
  \newtheorem{theorem}{Theorem}[section]
  \newtheorem*{theorem*}{Theorem}
  \newtheorem{proposition}[theorem]{Proposition}
  
  \newtheorem*{fact*}{Fact}
  \newtheorem{lemma}[theorem]{Lemma}
  
\theoremstyle{definition}
  \newtheorem{definition}[theorem]{Definition}
\theoremstyle{remark}
  \newtheorem*{remark}{Remark}
  \newtheorem*{acknowledgements}{Acknowledgements}
  \newtheorem{example}[theorem]{Example}
\numberwithin{equation}{section}

\newcommand{\R}{{\mathbb{R}}}

\renewcommand{\H}{{\mathbb{H}}}
\renewcommand{\S}{{\mathbb{S}}}

\renewcommand{\phi}{\varphi}
\renewcommand{\epsilon}{\varepsilon}

\newcommand{\inner}[2]{\left\langle{#1},{#2}\right\rangle}

\DeclareMathOperator{\Vol}{Vol}

\newcommand{\setm}{\,;\,}

\newcommand{\II}{\textrm{II}}
\newcommand{\Nabla}{\overline{\nabla}}
\newcommand{\pos}{\vb*{p}}
\newcommand{\dos}{\vb*{q}}
\newcommand{\vbH}{\vb*{H}}

\allowdisplaybreaks

\title[]{The volume of conformally flat manifolds as hypersurfaces in the light-cone}
\author[]{Riku Kishida}
\address[Riku Kishida]{Department of Mathematical and Computing Sciences, Tokyo Institute of Technology, Tokyo 152-8552, Japan}
\email{kishida.r.aa@m.titech.ac.jp}

\date{April 23, 2024.}

\begin{document}

\begin{abstract}
  In this paper, we focus on a conformally flat Riemannian manifold $(M^n,g)$ of dimension $n$ isometrically immersed into the $(n+1)$-dimensional light-cone $\Lambda^{n+1}$ as a hypersurface.
  We compute the first and the second variational formulas on the volume of such hypersurfaces.
  Such a hypersurface $M^n$ is not only immersed in $\Lambda^{n+1}$ but also isometrically realized as a hypersurface of a certain null hypersurface $N^{n+1}$ in the Minkowski spacetime, which is different from $\Lambda^{n+1}$.
  Moreover, $M^n$ has a volume-maximizing property in $N^{n+1}$.
\end{abstract}

\maketitle

\section*{Introduction}
Let $\R^{n+2}_1$ be the $(n+2)$-dimensional Minkowski spacetime of signature $(-+\dots+)$, and denote by $\inner{\,}{\,}$ the canonical Lorentzian inner product on $\R^{n+2}_1$.
The subset
\[\Lambda^{n+1}:=\left\{\vb*{x}\in\R^{n+2}_1\setm\langle\vb*{x},\vb*{x}\rangle=0\right\}\]
is called the $(n+1)$-dimensional \textit{light-cone}, which is a hypersurface in $\R^{n+2}_1$ and has a cone-like singularity at the origin and the induced metric on $\Lambda^{n+1}$ is degenerate at each point.

Let $M^n$ be an $n$-dimensional smooth manifold, and let $\pos:M^n\to\Lambda^{n+1}$ be an immersion.
If the induced metric on $M^n$ by $\pos$ is nondegenerate at each point in $M^n$, the map $\pos$ is called a \textit{spacelike hypersurface} in the light-cone $\Lambda^{n+1}$.
The following characterization of conformally flat Riemannian manifolds is known.
\begin{fact*}[Asperti-Dajczer~\cite{AD89}]
  Let $(M^n,g)$ be an $n(\ge 3)$-dimensional simply connected Riemannian manifold.
  Then, there exists an isometric immersion $\pos$ from $M^n$ to $\Lambda^{n+1}$ if and only if $(M^n,g)$ is conformally flat.
  Moreover, $\pos$ is unique up to orientation preserving isometries of $\Lambda^{n+1}$.
\end{fact*}

In the case of $n=2$, any $(M^2,g)$ can be locally isometrically immersed in $\Lambda^3$.
However, such immersions are not unique in general (cf. \cite{LUY11}).

Spacelike hypersurfaces in $\Lambda^{n+1}$ have the duality, which was shown by Espinar-G{\'{a}}lvez-Mira~\cite{EGM09}, Izumiya~\cite{Izumiya09} and Liu-Jung~\cite{LJ08} independently.
More precisely, there exists a unique map $\dos:M^n\to\Lambda^{n+1}$ satisfying
\begin{equation}\label{eq_rel_dual_map}
  \inner{\pos}{\dos}=1,\quad\inner{X}{\dos}=0\quad(X\in\mathfrak{X}(M^n)),
\end{equation}
where $\mathfrak{X}(M^n)$ denotes the set of tangent vector fields on $M^n$.
The map $\dos$ is called the \textit{dual map} with respect to $\pos$.
If $\dos$ is regarded as a vector field along $\pos$, then $\dos$ has the role of the unit normal vector field of $\pos$.
Moreover, the mean curvature function can be defined using $\dos$, and it coincides with a constant multiple of the scalar curvature $S$ of $(M^n,g)$.
In Liu-Umehara-Yamada~\cite{LUY11}, several geometric properties of $\dos$ are discussed in terms of a coherent tangent bundle over $M^n$.

We fix a simply connected conformally flat Riemannian manifold $(M^n,g)$ and assume that $M^n$ is isometrically immersed in $\Lambda^{n+1}(\subseteq\R^{n+2}_1)$.
We denote by $\pos:M^n\to\Lambda^{n+1}$ the corresponding isometric immersion. 
Then, the scalar curvature $S$ of $(M^n,g)$ is the solution of the first variation on the volume, which was shown by Andersson-Metzger~\cite{AM10}, Honda-Izumiya~\cite{HI15} and Liu-Jung~\cite{LJ08}.
In these papers, the second variational formula on the volume is shown.

Based on these results, in this paper, we compute the first and the second variational formulas and confirm the non-positivity of the second variation with respect to a certain variational vector field $X$ which is transversal to $\Lambda^{n+1}$ without any restriction of the dimension $n$.
At the same time, the geometrical meaning of the variational vector field $X$ is clarified.
As a consequence, we show that a conformally flat manifold $(M^n,g)$ induces a hypersurface $N^{n+1}$ in $\R^{n+2}_1$, which is called the \textit{null-space} (cf. Definition \ref{def_null_space}).
Then, we can regard $M^n$ as a hypersurface isometrically embedded in the null space $N^{n+1}$, not the light-cone $\Lambda^{n+1}$.
The null-space $N^{n+1}$ is a hypersurface ruled by $\dos$ in $\R^{n+2}_1$ and the induced metric on $N^{n+1}$ is degenerate everywhere.
We remark that $N^{n+1}$ is an L-complete null wave front in the sense of Akamine-Honda-Umehara-Yamada~\cite{AHUY22}.
When $S$ is identically zero, then $(M^n,g)$ can be considered as a zero mean curvature hypersurface with maximal volume in $N^{n+1}$.
The main result of this paper is as follows:

\begin{theorem*}\label{thm_intro}
  Let $(M^n,g)$ be an $n(\ge 2)$-dimensional oriented conformally flat Riemannian manifold and $\pos:M^n\to\Lambda^{n+1}$ an isometric immersion.
  Let $N^{n+1}$ be the null space with respect to $\pos$, and we can regard $\pos$ as an isometric embedding from $M^n$ to $N^{n+1}$.
  Then, the following statements are equivalent.
  \begin{enumerate}[\normalfont(a)]
    \item The scalar curvature $S$ of $(M^n,g)$ is identically zero.
    \item For any compactly supported variation of $\pos$ in $N^{n+1}$, the first variation on the volume is zero.
  \end{enumerate}
  Moreover, if one of these two conditions holds, the second variation is nonpositive for any compactly supported variation of $\pos$ in $N^{n+1}$.
\end{theorem*}

For the second variation of a zero mean curvature hypersurface in $\R^{n+1}$ defined on the domain $D(\subseteq\R^n)$, the minimality is achieved only if $D$ is sufficiently small.
The main theorem is much stronger assertion, that is, the maximality of a zero scalar curvature hypersurface in $\Lambda^{n+1}$ is independent of the choice of the domain $D$.

The paper is organized as follows:
In Section \ref{sec_1}, several properties of hypersurfaces in the light-cone are given as a preliminary.
In Section \ref{sec_2}, we show the variational formulas on the volume of hypersurfaces in $\Lambda^{n+1}$.
In Section \ref{sec_3}, we introduce the null-space $N^{n+1}$ induced by a given hypersurface $M^n$ in $\Lambda^{n+1}$.
In Section \ref{sec_4}, we present some examples of conformally flat Riemannian manifolds in $\Lambda^{n+1}$ with vanishing scalar curvatures and the null-spaces induced by them.

\section{Hypersurfaces in the light-cone}\label{sec_1}

Throughout this section, we assume that $(M^n,g)$ denotes an $n(\ge 2)$-dimensional conformally flat Riemannian manifold and $\pos:M^n\to\Lambda^{n+1}$ an isometric immersion.
Since the light-cone $\Lambda^{n+1}$ is a hypersurface in $\R^{n+2}_1$, the manifold $M^n$ can be considered as an immersed submanifold in $\R^{n+2}_1$ whose codimension is equal to~$2$.
Using $\pos$, we identify a tangent vector of $M^n$ with a vector of $\R^{n+2}_1$.
We denote by $TM^n$ and $NM^n$ the tangent bundle and the normal bundle of $M^n$, respectively.
In addition, let $\mathfrak{X}(M^n)$ be the set of tangent vector fields on $M^n$, and let $\Gamma(NM^n)$ be the set of normal vector fields along $\pos$.
Since $\inner{\pos}{\pos}=0$ and $\pos$ is orthogonal to the tangent space $TM^n$, there exists the dual map $\dos:M^n\to\Lambda^{n+1}$ with respect to $\pos$, satisfying \eqref{eq_rel_dual_map}.
Then, $\{\pos,\dos\}$ is the frame field of $NM^n$ defined on $M^n$.
Let $\nabla$ and $\Nabla$ be the Levi-Civita connections of $(M^n,g)$ and $\R^{n+2}_1$ respectively, and let $\nabla^\bot$ be the normal connection on $NM^n$.
The following proposition holds.
\begin{proposition}[\cite{AD89}]\label{prop_parallel}
  If we regard $\pos$ and $\dos$ as the normal vector field along $\pos$,
  they are parallel with respect to the normal connection $\nabla^\bot$.
\end{proposition}
\begin{proof}
  Since $\Nabla_X\pos=X\;(X\in\mathfrak{X}(M^n))$, $\pos$ is parallel with respect to $\nabla^\bot$.
  Moreover, $\dos$ is parallel, since the normal connection is compatible with the inner product on $NM^n$ induced from $\R^{n+2}_1$.
\end{proof}
Let $\II:\mathfrak{X}(M^n)\times\mathfrak{X}(M^n)\to\Gamma(NM^n)$ be the second fundamental form of $\pos$, that is, we define
\[\II(X,Y):=(\Nabla_XY)^\bot\quad(X,Y\in\mathfrak{X}(M^n)),\]
where $\bot$ denotes the projection into the normal bundle $NM^n$.
In addition, we define the map $A:\mathfrak{X}(M^n)\to\mathfrak{X}(M^n)$ as
\[\inner{AX}{Y}:=\inner{\II(X,Y)}{\dos}.\]
The map $A$ is called the \textit{shape operator} of $\pos$. Obviously, $A$ is a symmetric tensor on $M^n$.
We also define the normal vector field $\vbH\in\Gamma(NM^n)$ as
\[\vbH:=\sum_{i=1}^n\II(e_i,e_i),\]
where $\{e_1,\dots,e_n\}$ is an orthonormal frame field of $TM^n$.
We call $\vbH$ the \textit{mean curvature vector field} of $\pos$ in $\R^{n+2}_1$.
\begin{proposition}\label{prop_mean_curvature_vector}
  The mean curvature vector field $\vbH$ is written as
  \begin{equation}\label{eq_mean_curvature}
    \vbH=\tr(A)\pos-n\dos,
  \end{equation}
  where $\tr(A)$ represents the trace of the shape operator $A$.
\end{proposition}
\begin{proof}
  By using $\inner{\pos}{\dos}=1$, for each $X,Y\in\mathfrak{X}(M)$, we obtain
  \begin{equation}\begin{split}
    \II(X,Y)
    &=\inner{\II(X,Y)}{\dos}\pos+\inner{\II(X,Y)}{\pos}\dos\\
    &=\inner{AX}{Y}\pos-\inner{\Nabla_X\pos}{Y}\dos\\
    &=\inner{AX}{Y}\pos-\inner{X}{Y}\dos.\label{eq_second_form}
  \end{split}\end{equation}
  By contracting both sides of the above equation with respect to $g$, we obtain \eqref{eq_mean_curvature}.
\end{proof}

As in \cite{EGM09, Izumiya09, LJ08}, $\tr(A)$ can be expressed as a constant multiple of the scalar curvature $S$ of $(M^n,g)$.
\begin{proposition}\label{prop_scalar_is_mean}
  The scalar curvature $S$ of $(M^n,g)$ satisfies the following identity
  \[S=-2(n-1)\tr(A).\]
\end{proposition}
\begin{proof}
  We denote by $R$ the curvature tensor of $(M^n,g)$, that is, we define
  \[R(X,Y)Z:=\nabla_X\nabla_YZ-\nabla_Y\nabla_XZ-\nabla_{[X,Y]}Z\quad(X,Y,Z\in\mathfrak{X}(M^n)),\]
  where $[\,,\,]$ represents the Lie bracket of vector fields.
  By Gauss equation and \eqref{eq_second_form}, for each $X,Y,Z,W\in\mathfrak{X}(M^n)$ we obtain
  \begin{equation*}\begin{split}
    \inner{R(X,Y)Z}{W}
    &=\inner{\II(X,W)}{\II(Y,Z)}-\inner{\II(Y,W)}{\II(X,Z)}\\
    &=-\inner{X}{W}\inner{AY}{Z}-\inner{AX}{W}\inner{Y}{Z}\\
    &\qquad\qquad+\inner{Y}{W}\inner{AX}{Z}+\inner{AY}{W}\inner{X}{Z}.
  \end{split}\end{equation*}
  Let $\{e_1,\dots,e_n\}$ be an orthonormal frame field of $TM^n$.
  Then we obtain
  \begin{equation*}\begin{split}
    S&=\sum_{\substack{i,j=1\\i\neq j}}^n\inner{R(e_i,e_j)e_j}{e_i}
    =\sum_{\substack{i,j=1\\i\neq j}}^n\Bigl(-\inner{Ae_j}{e_j}-\inner{Ae_i}{e_i}\Bigr)\\
    &=-2(n-1)\sum_{i=1}^n\inner{Ae_i}{e_i}=-2(n-1)\tr(A).
  \end{split}\end{equation*}
  This completes the proof.
\end{proof}

\section{The variational formulas of hypersurfaces in the light-cone}\label{sec_2}

In this section, we compute the variational formulas on the volume of a hypersurface $\pos$ in $\Lambda^{n+1}$.
In Subsection \ref{subsec_1}, we compute the first variational formula with respect to $\pos$.
In Subsection \ref{subsec_2}, we make preparations for computing the second variational formula.
In Subsection \ref{subsec_3}, we compute the second variational formula with respect to $\pos$.

\subsection{The first variational formula for $\pos$}\label{subsec_1}

Let $(M^n,g)$ be an $n$-dimensional compact oriented conformally flat Riemannian manifold with boundary, and let $\pos:M^n\to\Lambda^{n+1}$ be an isometric immersion.
We consider a variation of $\pos$ with fixed boundary.
More precisely, let $\epsilon$ be a positive real number,
and consider a smooth map $F:(-\epsilon,\epsilon)\times M^n\to\R^{n+2}_1$ satisfying the following conditions:
\begin{itemize}
  \item For each $t\in (-\epsilon,\epsilon)$, the map $F(t,\cdot):M^n\to\R^{n+2}_1$ is a spacelike immersion.
  \item $F(0,x)=\pos(x)\quad(x\in M^n)$.
  \item For each point $x$ on the boundary of $M^n$, we have $F(\cdot,x)\equiv\pos(x)$.
\end{itemize}
For each $t\in (-\epsilon,\epsilon)$, we denote by $g_t$ the metric on $M^n$ induced from the spacelike immersion $F(t,\cdot):M^n\to\R^{n+2}_1$.
In addition, $dV_t$ denotes the volume form of the Riemannian manifold $(M^n,g_t)$, and $\Vol(t)$ is defined as the volume of $(M^n,g_t)$.
We denote by $X$ the variational vector field with respect to $F$, that is, we define $X:=\left.\frac{\partial }{\partial t}\right|_{t=0}F$.
It is a well-known fact that $\left.\frac{d}{dt}\right|_{t=0}dV_t=-\inner{X}{\vbH}dV_0$, where $\vbH$ is the mean curvature vector field of $\pos$.
By Proposition~\ref{prop_mean_curvature_vector}, the first variation on the volume can be written as
\begin{equation}\label{eq_first_variation_1}
  \left.\frac{d}{dt}\right|_{t=0}\Vol(t)=\int_{M^n}\Bigl(-\tr(A)\inner{X}{\pos}+n\inner{X}{\dos}\Bigr)dV_0,
\end{equation}
where $A$ represents the shape operator of $\pos$.
If the scalar curvature $S$ is identically zero (cf. Proposition \ref{prop_scalar_is_mean}), we obtain
\begin{equation}\label{eq_first_variation}
  \left.\frac{d}{dt}\right|_{t=0}\Vol(t)=n\int_{M^n}\inner{X}{\dos}dV_0.
\end{equation}
Therefore, in order to assure that the variation on the volume is critical, it is necessary to restrict the direction of the variational vector field.
For this purpose, we introduce the following.
\begin{definition}
  Let $F:(-\epsilon,\epsilon)\times M^n\to\R^{n+2}_1$ be a variation of $\pos$, and let $X$ be the variational vector field with respect to $F$.
  Then, $F$ is called an \textit{admissible variation} of $\pos$ if there exists a real-valued function $\phi$ defined on $M^n$ which satisfies $X=\phi\dos$,
  where $\dos$ is the dual map of $\pos$.
\end{definition}

If a variation $F$ is admissible, the second term of the integrand in \eqref{eq_first_variation_1} vanishes, so the following proposition holds.

\begin{proposition}\label{prop_first_variation_admissible}
  Let $(M^n,g)$ be an $n$-dimensional compact oriented conformally flat Riemannian manifold with boundary, and let $\pos:M^n\to\Lambda^{n+1}$ be an isometric immersion.
  For an admissible variation $F:(-\epsilon,\epsilon)\times M^n\to\R^{n+2}_1$ of $\pos$ with fixed boundary, the first variation on the volume is given by
  \[\left.\frac{d}{dt}\right|_{t=0}\textup{Vol}(t)=-\int_{M^n}\tr(A)\inner{X}{\pos}dV_0.\]
\end{proposition}

\subsection{The second variational formula in the general setting}\label{subsec_2}

We next consider the second variation.
In the case of a minimal (or maximal) submanifold in a usual setting, the second variation depends only on the variational vector field at $t=0$.
However, this is not true for hypersurfaces in the light-cone, since the mean curvature vector field $\vbH$ never takes a zero vector.
So we consider an ``extended'' variational vector field as follows:
\begin{definition}
  Let $(\overline{M}^m,\inner{\cdot}{\cdot})$ be a pseudo-Riemannian manifold, and let $(M^n,g)$ be a compact oriented Riemannian manifold with boundary.
  For a given variation $F:(-\epsilon,\epsilon)\times M^n\to\overline{M}^m$ of an isometric immersion $f:M^n\to\overline{M}^m$, we define the vector field $\overline{X}$ along $F$ as $\overline{X}:=\partial F/\partial t$.
  This $\overline{X}$ is called the \textit{extended variational vector field} of $F$.
\end{definition}
Here, the vector field $\Nabla_t\overline{X}$ along $f$ is defined as \[\left(\Nabla_t\overline{X}\right)_x:=\left.\frac{\partial}{\partial t}\right|_{t=0}\overline{X}_{(t,x)}=\left.\frac{\partial^2}{\partial t^2}\right|_{t=0}F(t,x)\quad(x\in M^n).\]
It is known that the second variation on the volume in the general case can be expressed as described in the following proposition.
\begin{proposition}\label{prop_second_variation_general}
  Let $(\overline{M}^m,\inner{\cdot}{\cdot})$ be a pseudo-Riemannian manifold, let $(M^n,g)$ be a compact oriented Riemannian manifold with boundary.
  For a variation $F$ of an isometric immersion $f:M^n\to\overline{M}^m$ with fixed boundary, the second variation on the volume is given by
  \begin{equation}\begin{split}\label{eq_second_variation_general}
    \left.\frac{d^2}{dt^2}\right|_{t=0}\textup{Vol}(t)
    =&\int_{M^n}\Biggl(
    \sum_{i=1}^n\inner{\overline{R}(X,e_i)X}{e_i}
    +\sum_{i=1}^n\inner{(\Nabla_{e_i}X)^\bot}{(\Nabla_{e_i}X)^\bot}\\
    &\qquad\qquad
    -\sum_{i,j=1}^n\inner{\Nabla_{e_i}X}{e_j}\inner{\Nabla_{e_j}X}{e_i}\\
    &\qquad\qquad
    +\sum_{i,j=1}^n\inner{\Nabla_{e_i}X}{e_i}\inner{\Nabla_{e_j}X}{e_j}
    -\inner{\Nabla_t\overline{X}}{\vbH}
  \Biggr)dV_0,
  \end{split}\end{equation}
  where we define as follows:
  \begin{itemize}
    \item $X$ is the variational vector field of $F$,
    \item $\overline{X}$ is the extended variational vector field of $F$,
    \item $\{e_1,\dots,e_n\}$ is an orthonormal frame field of $TM^n$,
    \item $\vbH$ is the mean curvature vector field of $f$,
    \item $\Nabla$ is the Levi-Civita connection of $(\overline{M}^m,\inner{\cdot}{\cdot})$, and
    \item $\overline{R}$ is the curvature tensor of $(\overline{M}^m,\inner{\cdot}{\cdot})$.
  \end{itemize}
\end{proposition}

The proof of Proposition~\ref{prop_second_variation_general} is given in Appendix \ref{appendix}.

\subsection{The second variational formula for $\pos$}\label{subsec_3}

In this subsection, we compute the second variational formula for hypersurfaces in $\Lambda^{n+1}$.
Let $F$ be an admissible variation of $\pos$ with fixed boundary, and let $X$ be the variational vector field of $F$.
We calculate the integrand function of the integral on $M^n$ in \eqref{eq_second_variation_general} for each term.
Since the Minkowski spacetime $\R^{n+2}_1$ is flat, the first term of the integrand in \eqref{eq_second_variation_general} is identically zero.
By the definition of an admissible variation, we can write $X=\phi\dos$ using some real-valued smooth function $\phi$ on $M^n$.
Then, for any tangent vector field $V\in\mathfrak{X}(M^n)$, we can calculate $\Nabla_VX$ as follows:
\begin{equation}\label{eq_first_derivative_vv}
  \Nabla_VX=(V\phi)\dos+\phi\Nabla_V\dos=(V\phi)\dos-\phi AV,
\end{equation}
where we used Proposition~\ref{prop_parallel} in the second equality above.
For a given $\{e_1,\dots,e_n\}$ as an orthonormal frame field of $TM^n$, the second term of the integrand in \eqref{eq_second_variation_general} can be calculated as follows:
\[\sum_{i=1}^n\inner{(\Nabla_{e_i}X)^\bot}{(\Nabla_{e_i}X)^\bot}=\sum_{i=1}^n\inner{(e_i\phi)\dos}{(e_i\phi)\dos}=0.\]
By using the fact that the shape operator $A$ is symmetric, the third term of the integrand in \eqref{eq_second_variation_general} can be written as
\begin{equation*}\begin{split}
  \sum_{i,j=1}^n\inner{\Nabla_{e_i}X}{e_j}\inner{\Nabla_{e_j}X}{e_i}
  &=\sum_{i,j=1}^n\inner{\phi Ae_i}{e_j}\inner{\phi Ae_j}{e_i}
  =\sum_{i,j=1}^n\phi^2\inner{Ae_i}{e_j}\inner{e_j}{Ae_i}\\
  &=\sum_{i=1}^n\phi^2\inner{Ae_i}{Ae_i}
  =\sum_{i=1}^n\phi^2\inner{A^2e_i}{e_i}
  =\phi^2\tr(A^2).
\end{split}\end{equation*}
In addition, the fourth term of the integrand in \eqref{eq_second_variation_general} can be expressed as
\[\sum_{i,j=1}^n\inner{\Nabla_{e_i}X}{e_i}\inner{\Nabla_{e_j}X}{e_j}=\sum_{i,j=1}^n\inner{\phi Ae_i}{e_i}\inner{\phi Ae_j}{e_j}=\phi^2\tr(A)^2.\]
We focus on the fifth term of the integrand in \eqref{eq_second_variation_general}, namely $\inner{\Nabla_t\overline{X}}{\vbH}$.
If the scalar curvature $S$ is identically zero, this term can be written as
\[\inner{\Nabla_t\overline{X}}{\vbH}=-n\inner{\Nabla_t\overline{X}}{\dos}.\]
The above calculations yield the following.
\begin{proposition}\label{prop_second_variation_ad}
  Let $(M^n,g)$ be an $n$-dimensional compact oriented conformally flat Riemannian manifold with boundary and $\pos:M^n\to\Lambda^{n+1}$ an isometric immersion.
  Let $F:(-\epsilon,\epsilon)\times M^n\to\R^{n+2}_1$ be an admissible variation of $\pos$ with fixed boundary.
  If the scalar curvature $S$ is identically zero, then the second variation on the volume is given by
  \[\left.\frac{d^2}{dt^2}\right|_{t=0}\textup{Vol}(t)=\int_{M^n}\Biggl(-\inner{X}{\pos}^2\tr(A^2)+n\inner{\Nabla_t\overline{X}}{\dos}\Biggr)dV_0,\]
  where $X$ and $\overline{X}$ are the variational vector field and the extended variational vector field of $F$, respectively.
\end{proposition}

If the term $\inner{\Nabla_t\overline{X}}{\dos}$ vanishes identically, the second variation does not take a positive value.
Namely, the variation on the volume is maximal.
This motivates us to consider the following variations, which are refinement of admissible variations.
\begin{definition}\label{def_characteristic}
  Let $F:(-\epsilon,\epsilon)\times M^n\to\R^{n+2}_1$ be an admissible variation of $\pos$.
  We denote by $X$ and $\overline{X}$ the variational vector field and the extended variational vector field of $F$, respectively.
  Then $F$ is called a \textit{characteristic variation} of $\pos$ if there exists a real-valued function $\phi$ defined on $M^n$ satisfying
  \[\Nabla_t\overline{X}=\phi\dos,\]
  where $\dos$ is the dual map of $\pos$.
\end{definition}
A simple example of characteristic variations is as follows:
\begin{example}\label{exam_char}
  Let $\phi$ be a real-valued function defined on $(-\epsilon,\epsilon)\times M^n$ satisfying $\phi(0,x)=0$ for each point $x$ of $M^n$.
  Let us consider the variation $F:(-\epsilon,\epsilon)\times M^n\to\R^{n+2}_1$ of $\pos$ described as
  \begin{equation}\label{eq_example_ch}
    F(t,x):=\pos(x)+\phi(t,x)\dos(x)\quad(t\in(-\epsilon,\epsilon),\;x\in M^n).
  \end{equation}
  Then, the variational vector field $X$ with respect to $F$ can be expressed as
  \[X_x=\left(\left.\frac{\partial}{\partial t}\right|_{t=0}\phi(t,x)\right)\dos(x)\quad(x\in M^n).\]
  This implies that $F$ is an admissible variation of $\pos$.
  Moreover, by calculating the second-order derivative of $F$ with respect to $t$, we obtain
  \[(\Nabla_t\overline{X})_x=\left(\left.\frac{\partial^2}{\partial t^2}\right|_{t=0}\phi(t,x)\right)\dos(x)\quad(x\in M^n).\]
  Therefore, $F$ is a characteristic variation of $\pos$.
\end{example}

In the case of characteristic variations, the second variational formula on the volume can be expressed as the following proposition.
\begin{proposition}\label{prop_second_variation_ch}
  Let $(M^n,g)$ be an $n$-dimensional compact oriented conformally flat Riemannian manifold with boundary, and let $\pos:M^n\to\Lambda^{n+1}$ be an isometric immersion.
  Let $F:(-\epsilon,\epsilon)\times M^n\to\R^{n+2}_1$ be a characteristic variation of $\pos$ with fixed boundary.
  If the scalar curvature is identically zero, then the second variation on the volume is given by
  \begin{equation}\label{eq_second_variation_lc}
    \left.\frac{d^2}{dt^2}\right|_{t=0}\textup{Vol}(t)=-\int_{M^n}\inner{X}{\pos}^2\tr(A^2)dV_0\le 0.
  \end{equation}
\end{proposition}
\begin{proof}
  By the definition of a characteristic variation, $F$ can be written as
  \[\Nabla_t\overline{X}=\phi\dos,\]
  where $\phi$ is a real-valued function on $M^n$.
  Since $F$ is also an admissible variation of $\pos$, from Proposition~\ref{prop_second_variation_ad} we obtain
  \[\left.\frac{d^2}{dt^2}\right|_{t=0}\textup{Vol}(t)=-\int_{M^n}\inner{X}{\pos}^2\tr(A^2)dV_0.\]
  Moreover, by the definition of a shape operator $A$, the $(1,1)$-tensor $A$ is symmetric.
  This implies that the value of $\tr(A^2)$ is nonnegative everywhere.
  Therefore, the following inequality holds.
  \[\left.\frac{d^2}{dt^2}\right|_{t=0}\textup{Vol}(t)\le 0.\]
\end{proof}

\begin{remark}
  The second variational formula stated in Proposition~\ref{prop_second_variation_ch} confirm the results given in \cite{AM10, HI15,LJ08}.
\end{remark}

\section{The null-spaces}\label{sec_3}

This section is the heart of this paper.
We will show that the second variational formula given in the previous section leads to the volume maximality of conformally flat manifolds with vanishing scalar curvature in a certain hypersurface in $\R^{n+2}_1$.
To interpret the variational formulas shown in the previous section geometrically, we introduce the notion of ``null-space'' as follows:

\begin{definition}\label{def_null_space}
  Let $(M^n,g)$ be an $n(\ge 2)$-dimensional conformally flat Riemannian manifold, and let $\pos:M^n\to\Lambda^{n+1}$ be an isometric immersion.
  We consider the map $\Phi_{\pos}:\R\times M^n\to\R^{n+2}_1$ given by
  \begin{equation}\label{eq_def_G}
    \Phi_{\pos}(t,x):=\pos(x)+t\dos(x)\quad(t\in\R,\;x\in M^n),
  \end{equation}
  which is called the \textit{extended ruled map} of $\pos$.
  We set $N^{n+1}:=\R\times M^n$, which is the domain of $\Phi_{\pos}$, and consider that $N^{n+1}$ is equipped with the metric $g_N$ defined as the pullback of $\inner{\,}{\,}$ by $\Phi_{\pos}$.
  This degenerate space $(N^{n+1},g_N)$ is called the \textit{null-space} with respect to $\pos$.
\end{definition}

Since $\Phi_{\pos}$ is an immersion on a tubular neighborhood of $\{0\}\times M^n$, we can regard $M^n$ as an embedded hypersurface in the null-space $N^{n+1}$.
Namely, let us consider a map $\iota:M^n\ni x\mapsto (0,x)\in N^{n+1}(=\R\times M^n)$.
Then the diagram

\begin{figure}[h!]%
 \begin{center}
       \includegraphics[height=1.8cm]{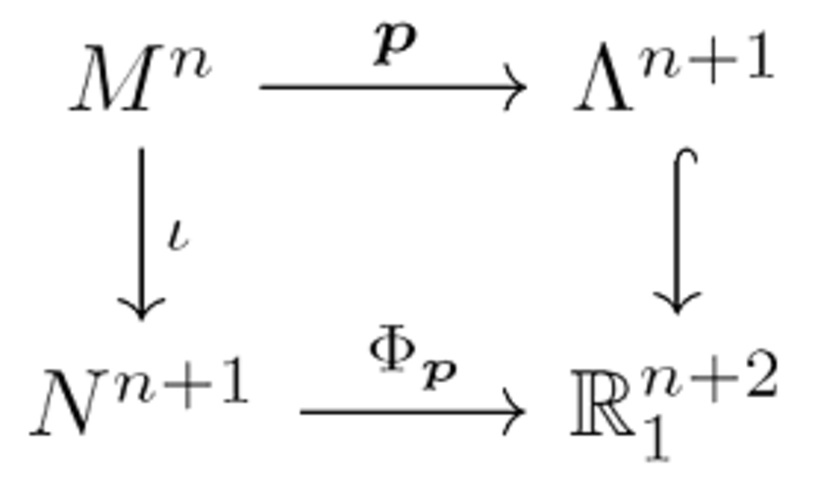} 
\end{center}
\end{figure}
\noindent
commutes and $\iota$ is an isometric embedding.
In particular, we can regard $M^n$ as a hypersurface isometrically embedded in $N^{n+1}$.
Therefore, the null-space $N^{n+1}$ have the role of an ambient space of $M^n$ which is different from the light-cone $\Lambda^{n+1}$.
Moreover, this space is a wave front with the completeness in the null direction.
Namely, the null hypersurface is foliated by entirety of light-like lines and the extended ruled map $\Phi_{\pos}$ is a wave front.
In other words, using the definition given in \cite{AHUY22}, the following proposition holds.

\begin{proposition}
  The extended ruled map $\Phi_{\pos}:N^{n+1}\to\R^{n+2}_1$ is an L-complete null wave front in $\R^{n+2}_1$.
\end{proposition}
\begin{proof}
  Using the criteria in \cite[Section 1]{AHUY22}, one can show that $\Phi_{\pos}$ is an L-complete null wave front.
\end{proof}

In addition, the extended ruled map $\Phi_{\pos}$ has the following interesting property.

\begin{proposition}
  If $t>0$, the image of $\Phi_{\pos}(t,\cdot)$ is contained as a hypersurface in the de Sitter spacetime whose constant curvature is $1/(2t)$.
  On the other hand, if $t<0$, the image of $\Phi_{\pos}(t,\cdot)$ is contained  as a hypersurface in the hyperbolic space whose constant curvature is $1/(2t)$.
\end{proposition}
\begin{proof}
  By \eqref{eq_rel_dual_map}, we obtain
  \[\inner{\Phi_{\pos}(t,x)}{\Phi_{\pos}(t,x)}=\inner{\pos(x)}{\pos(x)}+2t\inner{\pos(x)}{\dos(x)}+t^2\inner{\dos(x)}{\dos(x)}=2t,\]
  which proves the assertion.
\end{proof}

Since $M^n$ is a hypersurface in the null-space $N^{n+1}$, we can consider variations in $N^{n+1}$.
Let $G:(-\epsilon,\epsilon)\times M^n\to N^{n+1}$ be a variation of $\pos$ in $N^{n+1}$. More precisely, $G$ is a variation of the isometric embedding $\iota:M^n\ni x\mapsto (0,x)\in N^{n+1}$.
Therefore, $G(0,x)=(0,x)\;(x\in M^n)$ holds.

\begin{lemma}\label{lem_example}
  Let $\phi$ be a real-valued function on $M^n$.
  Then, there exists a variation $G:(-\epsilon,\epsilon)\times M^n\to N^{n+1}$ of $\pos$ satisfying that $F:=\Phi_{\pos}\circ G$ is a characteristic variation of $\pos$ in $\R^{n+2}_1$ and the variational vector field of $F$ coincides with $\phi\dos$.
\end{lemma}
\begin{proof}
  If we set $G(t,x):=(t\phi(x),x)$, then we have \[F(t,x)=\pos(x)+t\phi(x)\dos(x).\]
  This map $F$ is one of characteristic variations given in Example~\ref{exam_char}, and clearly its variational vector field is $\phi\dos$.
\end{proof}

The following proposition shows that arbitrary variations in $N^{n+1}$ can be replaced by a characteristic variation with fixed the volume function.

\begin{proposition}\label{prop_null_space_ch}
  Let $(M^n,g)$ be a conformally flat Riemannian manifold with boundary, and let $G:(-\epsilon,\epsilon)\times M^n\to N^{n+1}$ be a variation of an isometric immersion $\pos:M^n\to\Lambda^{n+1}$ with fixed boundary, where $N^{n+1}$ is the null-space of $\pos$.
  Then, there exists $\delta\in (0,\epsilon)$ and a characteristic variation $F:(-\delta,\delta)\times M^n\to\R^{n+2}_1$ of $\pos$ with fixed boundary satisfying the following condition.
  \begin{itemize}
    \item The volume function $\textup{Vol}_F(t)$ of $F$ coincides with $\textup{Vol}_G(t)$, that is, the following equality holds:
    \[\textup{Vol}_F(t)=\int_{M^n}dV_t=\int_{M^n}d\widetilde{V}_t=\textup{Vol}_G(t)\quad(t\in(-\delta,\delta)),\]
    where $dV_t$ and $d\widetilde{V}_t$ denote the volume forms of the metric induced by $F$ and $G$, respectively {\rm (cf. Figure 1)}.
  \end{itemize}
\end{proposition}

\begin{figure}[h!]%
 \begin{center}
       \includegraphics[height=2.7cm]{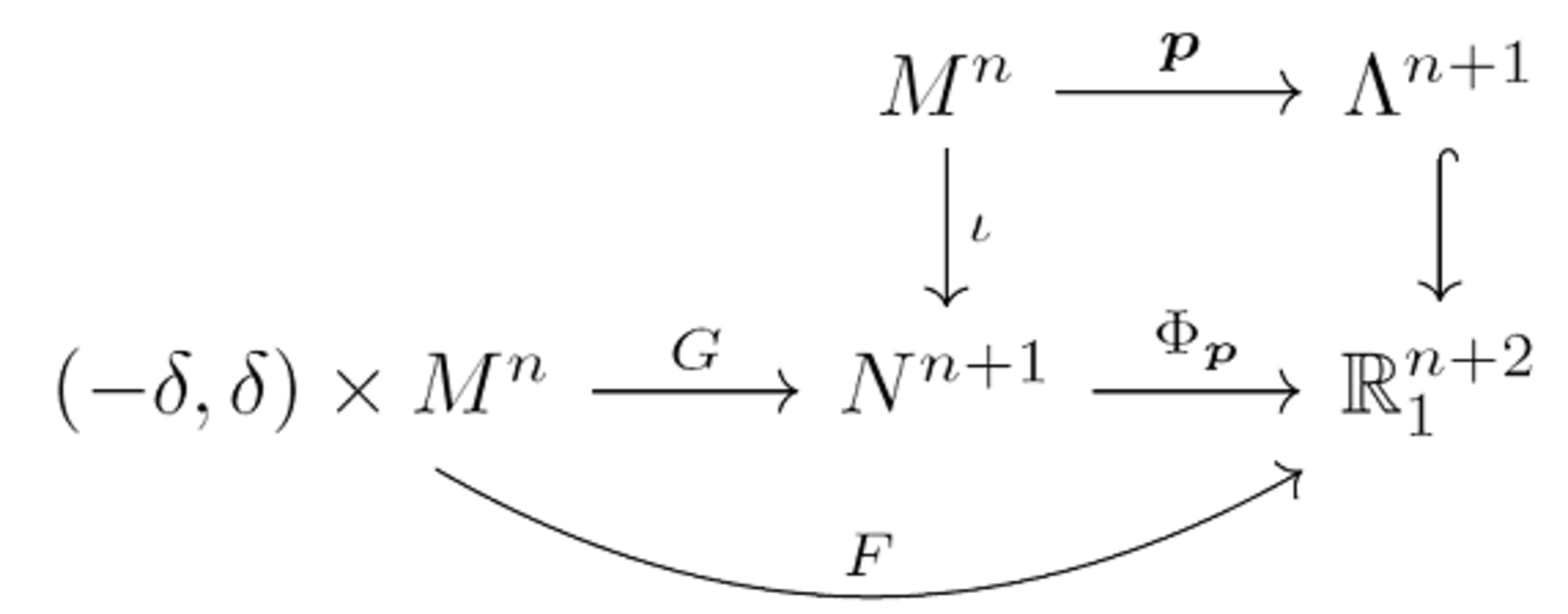} 
\end{center}
  \caption{The diagram related to $N^{n+1}$}%
\end{figure}

\begin{proof}
  We define two maps $\tau:(-\epsilon,\epsilon)\times M^n\to\R$ and $\alpha:(-\epsilon,\epsilon)\times M^n\to M^n$ as $G(t,x)=(\tau(t,x),\alpha(t,x))\in N^{n+1}$.
  Now consider the following map,
  \[\Phi:(-\epsilon,\epsilon)\times M^n\ni (t,x)\mapsto (t,\alpha(t,x))\in (-\epsilon,\epsilon)\times M^n.\]
  Since $G$ is a variation of $\pos$ with fixed boundary,
  \begin{itemize}
    \item $\tau(0,x)=0,\quad\alpha(0,x)=x\quad(x\in M^n)$, and
    \item $\alpha(t,x)=x\quad(x\in \partial M^n)$,
  \end{itemize}
  where $\partial M^n$ is the boundary of $M^n$.
  Therefore, for each point $x\in M^n$, the differential of $\Phi$ at $(0,x)$ is invertible.
  This implies that the map $\Phi$ is a diffeomorphism from some neighborhood of $(0,x)$ to another neighborhood of $(0,x)$.
  Since $M^n$ is compact, there exists $\delta\in(0,\epsilon)$ such that
  \[\left.\Phi\right|_{(-\delta,\delta)\times M^n}:(-\delta,\delta)\times M^n\to(-\delta,\delta)\times M^n\]
  is a diffeomorphism.
  Since $\alpha(t,\cdot):M^n\to M^n$ is a diffeomorphism for each $t\in (-\delta,\delta)$, we obtain a smooth map $\beta:(-\delta,\delta)\times M^n\to M^n$ which satisfies $\beta(t,\alpha(t,x))=x$.
  Here, we consider the following variation $F$,
  \[F:(-\delta,\delta)\times M^n\ni(t,x)\mapsto\pos(x)+\tau(t,\beta(t,x))\dos(x)\in\R^{n+2}_1.\]
  Since $F$ is a variation given in the Example \ref{exam_char}, $F$ is a characteristic variation of $\pos$ with fixed boundary.
  By the definition of $\beta$, it is clear that the image of $F$ coincides with that of $\Phi_{\pos}\circ G$.
  Therefore, both volume functions coincide.
\end{proof}

By Proposition~\ref{prop_null_space_ch}, we can show the main theorem.

\begin{proof}[Proof of Theorem in Introduction]
  Let $G$ be a compactly supported variation of $\pos$ in the null-space $N^{n+1}$ and suppose that the scalar curvature $S$ of $\pos$ is identically zero.
  By Proposition~\ref{prop_null_space_ch}, we can take a characteristic variation $F$ of $\pos$ in $\R^{n+2}_1$ such that the volume function of $F$ coincides with that of $G$.
  Since $F$ is characteristic, it follows from Propositions~\ref{prop_scalar_is_mean} and \ref{prop_first_variation_admissible} that the first variation on the volume with respect to $F$ is zero.
  Therefore, the first variation with respect to $G$ is also zero.
  
  Conversely, assuming that the first variation on the volume is zero for any compactly supported variation of $\pos$ in $N^{n+1}$.
  From Lemma~\ref{lem_example}, for any compactly supported real-valued function $\phi$, the integral $\int_{M^n}\tr(A)\phi dV_0$ is zero.
  Since we can show that $\tr(A)$ is identically zero by a well-known calculus, $S$ is also identically zero.
  
  Next, suppose that the scalar curvature is identically zero, and we will show that the second variation with respect to $G$ is nonpositive.
  As in the previous discussion, we take a characteristic variation $F$ in $\R^{n+2}_1$ such that the volume function of $F$ coincides with that of $G$ by Proposition~\ref{prop_null_space_ch}.
  Then, Proposition~\ref{prop_second_variation_ch} yields that the second variation with respect to $F$ is nonpositive.
  Therefore, the second variation with respect to $G$ is also nonpositive.
\end{proof}

\section{Examples}\label{sec_4}

\begin{example}
  One example of conformally flat Riemannian manifolds with zero scalar curvature is the Euclidean space $\R^n$.
  We can show by direct calculations that the following map $\pos$ gives an isometric immersion from $\R^n$ into $\Lambda^{n+1}$;
  \[\pos\colon\R^n\ni (x^1,\dots,x^n)\mapsto\left(\frac{1+\sum_{i=1}^n(x^i)^2}{2},\frac{-1+\sum_{i=1}^n(x^i)^2}{2},x^1,\dots,x^n\right)\in\Lambda^{n+1}.\]
  The dual $\dos:M^n\to\Lambda^{n+1}$ associated with $\pos$ takes the value $(-1,-1,0,\dots,0)$ identically.
  Therefore, the extended ruled map $\Phi_{\pos}$ with respect to $\pos$ is given by
  \[\Phi_{\pos}(t,x)=\left(\frac{1+\sum_{i=1}^n(x^i)^2}{2}-t,\frac{-1+\sum_{i=1}^n(x^i)^2}{2}-t,x^1,\dots,x^n\right)\;(t\in\R,x\in\R^n).\]
  Since $\Phi_{\pos}$ is injective and the image of $\Phi_{\pos}$ coincides with
  \[\{(y^1,y^2,\dots,y^{n+2})\in\R^{n+2}_1\setm y^1-y^2=1\},\]
  the null-space of $\pos$ can be regarded as a degenerated hyperplane in $\R^{n+2}_1$.
\end{example}

\begin{example}\label{example_2}
  The product manifold $\H^n\times\S^n$ is an example of conformally flat Riemannian manifolds with zero scalar curvature, where $\S^n$ is the unit sphere in $\R^{n+1}$ centered at the origin and $\H^n$ is defined as
  \[\H^n:=\{x=(x^1,\dots,x^{n+1})\in\R^{n+1}_1\setm\inner{x}{x}=-1,\;x^1>0\}.\]
  Since for each $(x,y)\in\H^n\times\S^n$ is a null vector of $\R^{2n+2}_1$, we can consider $\S^n\times\H^n$ as an embedded hypersurface in $\Lambda^{2n+1}$.
  This implies that $\S^n\times\H^n$ is a conformally flat Riemannian manifold.
  Let $\pos:\H^n\times\S^n\hookrightarrow\R^{2n+2}_1$ be the inclusion map. Then, the dual $\dos$ associated with $\pos$ is given by
  \[\dos(x,y)=\frac{1}{2}(-x,y)\quad(x\in\H^n,\;y\in\S^n).\]
  Therefore, the extended ruled map $\Phi_{\pos}$ with respect to $\pos$ can be written as
  \[\Phi_{\pos}(t,x,y)=\Bigl(\frac{2-t}{2}x,\frac{2+t}{2}y\Bigr)\quad(t\in\R,\;x\in\H^n,\;y\in\S^n).\]
  In this example, the scalar curvature of $\dos$ is also identically zero.
\end{example}

\begin{remark}
  When $n=2$, if $\pos:M^2\to\Lambda^{3}$ has zero scalar curvature, the dual map $\dos$ of $\pos$ has also zero scalar curvature on the regular set of $\dos$ (cf. \cite[Proposition~5.4]{LUY11}).
  In general, even if $\pos:M^n\to\Lambda^{n+1}$ has zero scalar curvature, it does not necessarily mean that the dual map $\dos$ of $\pos$ also does so, although the scalar curvature of $\dos$ given in Example~\ref{example_2} is identically zero.
\end{remark}

\begin{acknowledgements}
  The author would like to express my deep gratitude to Atsufumi Honda, Shunsuke Ichiki and Masaaki Umehara for many valuable discussions.
\end{acknowledgements}

\appendix
\section{Proof of Proposition~\ref{prop_second_variation_general}}\label{appendix}

In this appendix, we show Proposition~\ref{prop_second_variation_general}.
For the proof, we relied mainly on \cite[Section 6]{SW01}.

\begin{proof}[Proof of Proposition~\ref{prop_second_variation_general}]
  Fixed a point $x$ in $M^n$ and let $\{e_1,\dots,e_n\}$ be an orthonormal frame field of $TM^n$ satisfying
  \begin{equation}\label{eq_regular_onb}
    (\nabla_{e_i}e_j)_x=0\quad(i,j\in\{1,\dots,n\}),
  \end{equation}
  where $\nabla$ is the Levi-Civita connection of $(M^n,g)$.
  For each $t\in(-\epsilon,\epsilon)$ and $i\in\{1,\dots,n\}$, we can regard $e_i$ as a tangent vector field on $\{t\}\times M^n$, and we set $e_{i,t}(x):=dF_{(t,x)}(e_i)\;(x\in M^n)$.
  Then, by well-known facts on the first variation on the volume, we have
  \[\frac{d}{dt}dV_t=\sum_{i,j=1}^ng_t^{ij}\inner{\Nabla_{e_{i,t}}\overline{X}}{e_{j,t}}dV_t,\]
  where $\bigl(g_t^{ij}\bigr)$ is the inverse matrix of $\bigl(\inner{e_{i,t}}{e_{j,t}}\bigr)$.
  So, we obtain
  \begin{equation*}\begin{split}
    &\left.\frac{d^2}{dt^2}\right|_{t=0}dV_t
    =\left.\frac{d}{dt}\right|_{t=0}\left(\sum_{i,j=1}^ng_t^{ij}\inner{\Nabla_{e_{i,t}}\overline{X}}{e_{j,t}}\right)dV_0
    +\sum_{i=1}^n\inner{\Nabla_{e_{i}}X}{e_{i}}\left.\frac{d}{dt}\right|_{t=0}dV_t\\
    &\qquad=\left.\frac{d}{dt}\right|_{t=0}\left(\sum_{i,j=1}^ng_t^{ij}\inner{\Nabla_{e_{i,t}}\overline{X}}{e_{j,t}}\right)dV_0+\sum_{i,j=1}^n\inner{\Nabla_{e_i}X}{e_i}\inner{\Nabla_{e_j}X}{e_j} dV_0.
  \end{split}\end{equation*}
  If we set
  \[\psi(t):=\sum_{i,j=1}^ng_t^{ij}\inner{\Nabla_{e_{i,t}}\overline{X}}{e_{j,t}},\]
  we then have the following equality at $x$;
  \begin{align*}
    \left.\frac{d}{dt}\right|_{t=0}\psi
    &=\sum_{i,j=1}^n\left(\left.\frac{d}{dt}\right|_{t=0}g_t^{ij}\right)\inner{\Nabla_{e_{i}}X}{e_{j}}
    +\sum_{i=1}^n\left.\frac{d}{dt}\right|_{t=0}\inner{\Nabla_{e_{i,t}}\overline{X}}{e_{i,t}}\\
    &=\sum_{i,j=1}^n\left(-\left.\frac{d}{dt}\right|_{t=0}\inner{e_{i,t}}{e_{j,t}}\right)\inner{\Nabla_{e_{i}}X}{e_{j}}
    +\sum_{i=1}^n\inner{\Nabla_X\Nabla_{e_{i,t}}X}{e_{i,t}}\\
    &\qquad+\sum_{i=1}^n\inner{\Nabla_{e_i}X}{\Nabla_Xe_{i,t}}\\
    &=-\sum_{i,j=1}^n\inner{\Nabla_{e_i}X}{e_{j}}^2
    -\sum_{i,j=1}^n\inner{e_{i}}{\Nabla_{e_j}X}\inner{\Nabla_{e_i}X}{e_j}\\
    &\qquad+\sum_{i=1}^n\inner{\Nabla_X\Nabla_{e_{i,t}}\overline{X}}{e_i}
    +\sum_{i=1}^n\inner{\Nabla_{e_i}X}{\Nabla_{e_i}X}\\
    &=-\sum_{i=1}^n\inner{(\Nabla_{e_i}X)^\top}{(\Nabla_{e_i}X)^\top}
    -\sum_{i,j=1}^n\inner{e_{i}}{\Nabla_{e_j}X}\inner{\Nabla_{e_i}X}{e_j}\\
    &\qquad
    +\sum_{i=1}^n\inner{\overline{R}(X,e_i)X}{e_i}
    +\sum_{i=1}^n\inner{\Nabla_{e_i}\Nabla_t\overline{X}}{e_i}
    +\sum_{i=1}^n\inner{\Nabla_{e_i}X}{\Nabla_{e_i}X}\\
    &=\sum_{i=1}^n\inner{(\Nabla_{e_i}X)^\bot}{(\Nabla_{e_i}X)^\bot}
    -\sum_{i,j=1}^n\inner{e_{i}}{\Nabla_{e_j}X}\inner{\Nabla_{e_i}X}{e_j}\\
    &\qquad
    +\sum_{i=1}^n\inner{\overline{R}(X,e_i)X}{e_i}
    +\sum_{i=1}^n\inner{\Nabla_{e_i}\Nabla_t\overline{X}}{e_i}\\
    &=\sum_{i=1}^n\inner{(\Nabla_{e_i}X)^\bot}{(\Nabla_{e_i}X)^\bot}
    -\sum_{i,j=1}^n\inner{e_{i}}{\Nabla_{e_j}X}\inner{\Nabla_{e_i}X}{e_j}\\
    &\qquad
    +\sum_{i=1}^n\inner{\overline{R}(X,e_i)X}{e_i}
    -\sum_{i=1}^n\inner{\Nabla_t\overline{X}}{\II(e_i,e_i)}
    +\sum_{i=1}^ne_i\inner{\Nabla_t\overline{X}}{e_i}\\
    &=\sum_{i=1}^n\inner{(\Nabla_{e_i}X)^\bot}{(\Nabla_{e_i}X)^\bot}
    -\sum_{i,j=1}^n\inner{e_{i}}{\Nabla_{e_j}X}\inner{\Nabla_{e_i}X}{e_j}\\
    &\qquad
    +\sum_{i=1}^n\inner{\overline{R}(X,e_i)X}{e_i}
    -\inner{\Nabla_t\overline{X}}{\vbH}
    +\sum_{i=1}^ne_i\inner{\Nabla_t\overline{X}}{e_i},
  \end{align*}
  where $\top$ denotes the projection into the tangent bundle $TM^n$.
  Note that we used $[X,e_i]=0$ in the fourth equality. In addition, the sixth equality follows from \eqref{eq_regular_onb}.
  By the divergence theorem, we can obtain
  \[\int_{M^n}\sum_{i=1}^ne_i\inner{\Nabla_t\overline{X}}{e_i}dV_0=0.\]
  Thus, the second variation can be written as \eqref{eq_second_variation_general}.
\end{proof}

\end{document}